\documentclass[oneside]{amsart}
\usepackage[latin9]{inputenc}
\usepackage{amstext}
\usepackage{amsthm}
\usepackage{amssymb}
\usepackage{esint}

\makeatletter
\theoremstyle{plain}
\newtheorem{thm}{\protect\theoremname}[section]
  \theoremstyle{plain}
  \newtheorem{assumption}[thm]{\protect\assumptionname}
  \theoremstyle{remark}
  \newtheorem{rem}[thm]{\protect\remarkname}
  \theoremstyle{definition}
  \newtheorem{example}[thm]{\protect\examplename}
  \theoremstyle{plain}
  \newtheorem{lem}[thm]{\protect\lemmaname}
  \theoremstyle{definition}
  \newtheorem{defn}[thm]{\protect\definitionname}

\usepackage{hyperref}
\usepackage{url}
\usepackage{cite}

\renewcommand{\div}{\operatorname{div}}
\newcommand{\sgn}{\operatorname{sgn}}

\numberwithin{equation}{section}

\newcommand{\Bcal} {{\mathcal B}}
\newcommand{\Ccal} {{\mathcal C}}
\newcommand{\Dcal} {{\mathcal D}}

\newcommand{\Fcal} {{\mathcal F}}

\newcommand{\Ocal} {{\mathcal O}}

\newcommand{\R}{\mathbb{R}}

\renewcommand{\P}{\mathbb{P}}
\newcommand{\E}{\mathbb{E}}

\subjclass[2010]{Primary: 35K55, 35K92, 60H15; Secondary: 49J40, 49J45, 60H25;}

\makeatother

  \providecommand{\assumptionname}{Assumption}
  \providecommand{\definitionname}{Definition}
  \providecommand{\examplename}{Example}
  \providecommand{\lemmaname}{Lemma}
  \providecommand{\remarkname}{Remark}
\providecommand{\theoremname}{Theorem}

\begin{document}

\title[Nonlinear SPDE with gradient Stratonovich noise]{Nonlinear stochastic partial differential equations with singular
diffusivity and gradient Stratonovich noise}

\author{Ioana Ciotir}

\address{Normandie Univ\\
INSA Rouen, LMI\\
76000 Rouen\\
France}

\email{ioana.ciotir@insa-rouen.fr}

\author{Jonas M. Tölle}

\address{Department of Mathematics and Systems Analysis\\
School of Science\\
Aalto University\\
P.O. Box 11100\\
FI-00076 Aalto\\
Finland}

\email{jonas.tolle@aalto.fi}

\keywords{Stochastic variational inequality, singular diffusivity, nonlinear
singular-degenerate stochastic partial differential equation, total
variation flow, multiplicative gradient-type Stratonovich noise, singular
$p$-Laplace, Killing vector field, binary tomography, stochastic
methods in mathematical image processing.}

\thanks{I. C. was partially supported by the European Union with the European
regional development fund (ERDF) and by the Haute-Normandie Regional
Council (M2NUM Project). J. M. T. gratefully acknowledges funding
granted by the CRC 701 ``Spectral Structures and Topological Methods
in Mathematics'' (Bielefeld) of the German Research Foundation (DFG)}

\date{\today}
\begin{abstract}
We study existence and uniqueness of a variational solution in terms
of stochastic variational inequalities (SVI) to stochastic nonlinear
diffusion equations with a highly singular diffusivity term and multiplicative
Stratonovich gradient-type noise. We derive a commutator relation
for the unbounded noise coefficients in terms of a geometric Killing
vector condition. The drift term is given by the total variation flow,
respectively, by a singular $p$-Laplace-type operator. We impose
nonlinear zero Neumann boundary conditions and precisely investigate
their connection with the coefficient fields of the noise. This solves
an open problem posed in {[}Barbu, Brze\'{z}niak, Hausenblas, Tubaro;
Stoch. Proc. Appl., 123 (2013){]} and {[}Barbu, Röckner; J. Eur. Math.
Soc., 17 (2015){]}.
\end{abstract}

\maketitle

\section{Introduction}

We consider existence and uniqueness of solutions to the following
(multi-valued) nonlinear Stratonovich stochastic diffusion equation
in $L^{2}(\mathcal{{O}})$, 
\begin{equation}
\left\{ \begin{aligned}dX_{t} & \in\div\left[\sgn\left(\nabla X_{t}\right)\right]dt+\sum_{i=1}^{N}\langle b_{i},\nabla X_{t}\rangle\,\circ d\beta_{t}^{i}, &  & \text{in\ensuremath{\;(0,T)\times\mathcal{{O}}},}\\
X_{0} & =x, &  & \text{in}\ensuremath{\;\mathcal{{O}}},\\
\dfrac{\partial X_{t}}{\partial\nu} & =0, &  & \text{on}\;(0,T)\times\mathcal{\partial{O}},
\end{aligned}
\right.\label{eq:main_equation}
\end{equation}
where $\mathcal{O}$ is an open, bounded domain in $\mathbb{R}^{d}$,
$d\ge2$, with (sufficiently) smooth boundary such that $\mathcal{{O}}$
or $\partial\mathcal{{O}}$ is convex. Here, for $N\ge1$, $b_{i}:\overline{\Ocal}\to\R^{d}$,
$1\le i\le N$ are ``coefficient fields'' and $\beta=(\beta^{1},\ldots,\beta^{N})$
denotes an $N$-dimensional Brownian motion on a filtered (normal)
probability space $\left(\Omega,\mathcal{F},\left\{ \mathcal{F}_{t}\right\} _{t\ge0},\mathbb{P}\right)$.
The initial datum is chosen as $x\in L^{2}(\Ocal)$, or, more generally,
as $x\in L^{2}(\Omega,\Fcal_{0},\P;L^{2}(\Ocal))$. Here, $\nu$ denotes
the outer unit normal on $\partial\Ocal$. The multi-valued graph
$\xi\mapsto\sgn(\xi)$ from $\mathbb{R}^{d}$ into $2^{\mathbb{R}^{d}}$
is defined by 
\[
\sgn(\xi):=\left\{ \begin{array}{ll}
\dfrac{\xi}{\left\vert \xi\right\vert }, & \text{, if }\xi\neq0,\\
\left\{ \zeta\in\mathbb{R}^{d}\vert\left\vert \zeta\right\vert \leq1\right\}  & \text{, if }\xi=0,
\end{array}\right.
\]
for all $\xi\in\mathbb{R}^{d}$. Because of the multi-valued diffusivity
term, the equation becomes formally a \emph{stochastic evolution inclusion},
as have been studied e.g. in \cite{Kree82,GessToelle14,Gess:2015gw}.
We denote by $\left\vert \cdot\right\vert $ the Euclidean norm of
$\mathbb{R}^{d}$, and by $\langle\cdot,\cdot\rangle$ the Euclidean
scalar product of $\R^{d}$.

Set 
\begin{equation}
\mathbf{b}:=\begin{pmatrix}b_{1}\\
\vdots\\
b_{N}
\end{pmatrix}:\overline{\Ocal}\to\R^{N\times d},\label{eq:B_definition}
\end{equation}
and denote by $\mathbf{b}^{\ast}$ its transpose. We have that equation
\eqref{eq:main_equation} is formally equivalent to the Itô stochastic
partial differential equation, 
\begin{equation}
\left\{ \begin{aligned}dX_{t} & \in\div\left[\sgn\left(\nabla X_{t}\right)\right]dt+\dfrac{1}{2}\div[\mathbf{b}^{\ast}\mathbf{b}\nabla X_{t}]\,dt+\langle\mathbf{b}\nabla X_{t},d\beta_{t}\rangle, &  & \text{in\ensuremath{\;(0,T)\times\mathcal{{O}}},}\\
X_{0} & =x, &  & \text{in}\ensuremath{\;\mathcal{{O}}},\\
\dfrac{\partial X_{t}}{\partial\nu} & =0. &  & \text{on}\;(0,T)\times\mathcal{\partial{O}}.
\end{aligned}
\right.\label{eq:main_equation-1}
\end{equation}

A similar equation was studied in \cite{BBHT} for the case of a dissipative
drift, using the method of Brézis-Ekeland's variational principle\footnote{Equation \eqref{eq:main_equation-1} with Dirichlet boundary conditions
(instead of Neumann boundary conditions) is also being investigated
in \cite{RoecknerMunteanu2016}. However, the preprint of \cite{RoecknerMunteanu2016}
became publicly available after our revised work was submitted for
publication. We point out that the method used in \cite{RoecknerMunteanu2016}
is different from ours.}. On the other hand, equations with singular drift of the same form
have been studied in \cite{singular1,VBMR,BR15,Kim:2012en} for additive
and multiplicative bounded noise, respectively. See \cite{Krause:2014ie}
for a multiplicative Stratonovich stochastic equation with a similar
drift term. Those results do not apply to our case since the noise
coefficient 
\begin{equation}
u\mapsto\langle\mathbf{b}\nabla u,\cdot\rangle\label{eq:noise_coeff_intro}
\end{equation}
is not bounded on the state space $L^{2}\left(\mathcal{O}\right)$.
In \cite{GessRoe2014}, existence and uniqueness as well as regularity
have been investigated for the stochastic mean curvature flow with
unbounded noise. The methods used are related to ours, however, the
structure of the equation prevents a direct application to our situation.

Additionally, in our main Theorem \ref{thm:main_thm}, we will derive
existence and uniqueness results also for the singular $p$-Laplace
equations with $p\in(1,2)$, 
\begin{equation}
\left\{ \begin{aligned}dX_{t} & =\div\left[|\nabla X_{t}|^{p-2}\nabla X_{t}\right]dt+\sum_{i=1}^{N}\langle b_{i},\nabla X_{t}\rangle\,\circ d\beta_{t}^{i}, &  & \text{in\ensuremath{\;(0,T)\times\mathcal{{O}}},}\\
X_{0} & =x, &  & \text{in}\ensuremath{\;\mathcal{{O}}},\\
\dfrac{\partial X_{t}}{\partial\nu} & =0. &  & \text{on}\;(0,T)\times\mathcal{\partial{O}}.
\end{aligned}
\right.\label{eq:main_equation-1-1}
\end{equation}

Due to the lack of strong coercivity of the drift operator, we shall
employ so-called stochastic variational inequalities (SVI), with the
aim to construct solutions to \eqref{eq:main_equation-1-1} in a weak
variational sense. Even for bounded noise, singular equations of the
above type are generally not known to satisfy an Itô integral equation
--- not even in the (analytically) weak sense. Compare with \cite{singular1,VBMR,GessRoe2014,GessToelle15}
for related works employing SVI-frameworks. Using a rough path approach,
equations with similar noise were studied in \cite{CFO11,Friz:2015}.
A similar equation with linear drift is investigated in \cite{Beck:2014}.
We would like to point out, that the solutions of the work at hand
are strong solutions in the probabilistic sense, meaning, in particular,
that the solutions are functions of the given Brownian motion.

The natural energy space for the (Neumann) total variation flow, the
$p$-Laplace, respectively, would be $BV(\Ocal)$, the space of bounded
variation functions, respectively, the Sobolev space $W^{1,p}(\Ocal)$.
However, on the level of approximations, we shall work on the smaller
space $H^{1}(\Ocal)$. One reason is, that we are using viscosity
approximations, namely, we are adding a regularization term $\varepsilon\Delta$,
and taking $\varepsilon\searrow0$. In particular, this allows us
to consider the gradient-type SPDE for the borderline case of a monotone
drift operator ($p=1$) which cannot be treated within the scope of
\emph{reflexive} Gelfand triples ($p>1$), as e.g. has been done in
\cite{BBHT,BR15} for Dirichlet boundary conditions.

Another property, necessary for our arguments, is the mutual commutation
behavior of the diffusion coefficients, as well as the question of
commutation with the Neumann Laplace --- in order to obtain these,
we introduce a condition from differential geometry, similar to the
notion of Killing vector fields, see Assumption \ref{assu:well} and
Appendix \ref{sec:Vector-fields} below. In this context, we prove
that, under our assumptions, the first-order partial differential
operator \eqref{eq:noise_coeff_intro}, which corresponds to an infinitesimal
vector field action, preserves Neumann boundary conditions, see Lemma
\ref{lem:Neumann-invariance} below. 

According to \cite{brain}, the interest in studying this type of
equation comes from its use for simulations in the tomographic reconstruction
problem, which has several applications, for instance in medical imaging
and general image processing.

More precisely, the binary tomography methods are proposed in \cite{kuba}
as a simpler inverse problem of reconstruction. Being still an ill-posed
problem, it needs to be regularized, and this may be done for instance
with the total variation (T.V.) regularization. In order to numerically
solve the problem, a fast and efficient T.V./$L^{2}$ minimization
algorithm based on the ``Alternate Direction of Minimization Method''
(A.D.M.M.) has been proposed in \cite{afonso,wang-yang}. Finally,
a singular stochastic diffusion equation with gradient dependent noise
is used to refine the solution obtained by the A.D.M.M. algorithm,
see also the related Example \ref{exa:2D} below. The time dependent
(deterministic) T.V. image restoration problem has been studied e.g.
in \cite{Bogelein:2015kg}. We refer to \cite{Hakula:2014} and the
references therein for a stationary stochastic approach.

Therefore, the present work gives rigorous theory to support the use
of this kind of equation for numerical results such that those in
\cite{brain}. However, the authors of \cite{brain} are posing the
problem for an Itô-equation instead of a Stratonovich one, see also
\cite{Wang:2014gn}.

Another possible interest of studying stochastic differential equations
perturbed by this type of noise comes from the applications in modes\ of
turbulence (see \cite{turbulence}).

\subsection*{Discussion of an approach via transformation}

Following the classical works \cite{Doss76,Sussmann78}, we can also
think of an alternative access to our equation, which, however, must
fail even on a heuristic level. Here, we shall briefly discuss this
approach and point out the difficulties.

Let $y\in L^{2}(\Ocal)$ and consider the following deterministic
PDE 
\[
dY_{t}(\xi)\in\div[\operatorname{sgn}(\nabla Y_{t}(\xi))]\,dt,\quad Y_{0}(\xi)=y(\xi),\quad t\in(0,T],\;\xi\in\Ocal,
\]
where we impose Neumann boundary conditions. For initial datum $y\in H^{1}(\Ocal)$,
a unique weak solution in the Gelfand triple $H^{1}\subset L^{2}\subset H^{-1}$
was constructed in \cite[Theorem 2.6]{GessToelle14}. For initial
conditions $y\in L^{2}(\Ocal)$, see \cite{Andreu:2001vn,andreu2012parabolic,deLeon:2015cs}.

Let $\mathbf{b}$ be as in \eqref{eq:B_definition}, and assume merely
that $b_{i}\in C^{1}(\overline{\Ocal};\R^{d})$ for $1\le i\le N$.
For $t\in[0,T]$, $\xi\in\Ocal$, $\omega\in\Omega$, define
\begin{equation}
X_{t}(\xi)(\omega):=Y_{t}(\xi+\mathbf{b}^{\ast}(\xi)\beta_{t}(\omega)),\quad X_{0}(\xi)=y(\xi).\label{eq:transform}
\end{equation}
A similar transformation approach can be found in \cite{Beck:2014}
for linear equations and in \cite{Friz:2015} for the case of conservation
laws. See also \cite{BenthDeckPotthoffStreit:1998} for other nonlinear
SPDEs treated by this transformation.

Assume for a while, that we have a pathwise Itô formula available
(that is, for $\omega\in\Omega$, fixed), ignoring the lack of regularity
of $(x,t)\mapsto Y_{t}(\xi+\mathbf{b}^{\ast}(\xi)x)=:F(x,t)$ for
a moment:
\[
F(\beta_{t},t)=F(0,0)+\sum_{i=1}^{N}\int_{0}^{t}\partial_{x_{i}}F(\beta_{s},s)\,\circ d\beta_{s}^{i}+\int_{0}^{t}\partial_{t}F(\beta_{s},s)\,ds,
\]
see \cite{Protter:1995be,Bardina:1997fi,eisenbaum:2001,Bardina:2007kb}.
By the chain rule, we would obtain that for $dt$-a.e. $t\in[0,T]$,
possibly outside an exceptional subset of $\Ocal$, 
\[
X_{t}\in y+\int_{0}^{t}\div[\operatorname{sgn}(\nabla X_{s})]\,ds+\sum_{i=1}^{N}\sum_{j=1}^{d}\int_{0}^{t}b_{i}^{j}\partial_{\xi_{j}}X_{s}\,\circ d\beta_{s}^{i},
\]
which is a pathwise representation of equation \eqref{eq:main_equation}.
The Stratonovich correction term is formally given by
\[
\frac{{1}}{2}[\nabla F(\beta,\cdot),\beta]_{t}=\frac{1}{2}\sum_{k=1}^{d}\sum_{i=1}^{N}\int_{0}^{t}\partial_{\xi_{k}}(\langle b_{i},\nabla X_{s}\rangle)\,ds=\frac{1}{2}\int_{0}^{t}\div[\mathbf{b}^{\ast}\mathbf{b}\nabla X_{s}]\,ds,
\]
where $t\mapsto[\cdot,\cdot]_{t}$ denotes the quadratic covariation
process, compare with \cite{Protter:1995be,eisenbaum:2001}.

Even if one finds a way to deal with the measurability issues, the
direct application of this approach must fail due to the lack of regularity,
since, according to \cite{Yazhe:1986ky,Yazhe:1992do,DiBe}, good Hölder
estimates for the solution (and for the gradient of the solution)
to the parabolic $p$-Laplace equation usually hold only if $p>\frac{2d}{d+2}$,
thus sorting out the total variation flow.

\subsection*{Organization of the paper}

After a brief part on notational conventions of this work, we shall
give our assumptions and discuss the resulting properties of the noise
coefficient operators \eqref{eq:noise_coeff_intro} in Section \ref{sec:Hypo_commutation}
--- in particular, we establish the commutation relations which we
shall need subsequently. Our notion of SVI-solutions (to equations
with gradient-type multiplicative Stratonovich noise) is provided
in Section \ref{sec:SVI}. In Section \ref{sec:Existence-and-uniqueness},
we shall first derive a useful a priori estimate in $H^{2}(\Ocal)$
and after that go through several approximation steps necessary for
proving the existence of a solution. The uniqueness of SVI solutions
is proved in Subsection \ref{sub:Uniqueness}. For the reader's convenience,
we shall provide some results on Killing vector fields in the appendix.

\subsection*{Notation}

We shall recall a few standard definitions and fix notation which
will be used later.

We set $V:=H^{1}(\Ocal)=W^{1,2}(\Ocal)$, the standard first order
square integrable Sobolev space and $H:=L^{2}(\Ocal)$, the Hilbert
space of (classes of) square integrable functions with respect to
the Lebesgue measure. We also consider the second order square integrable
Sobolev space $H^{2}(\Ocal)=W^{2,2}(\Ocal)$. We shall write $H^{2}$,
$H^{1}$, $L^{2}$, and so on, if the context is clear. We denote
the inner product in $H$ by $(\cdot,\cdot)_{H}$. $V^{\ast}$ denotes
the topological dual of $V$ with dualization denoted by $\langle\cdot,\cdot\rangle$.
Let $W^{1,p}(\Ocal)$ be the usual first order $p$-integrable Sobolev
space. For $u\in L^{1}(\Ocal)$ we define the total variation semi-norm
by 
\[
\|u\|_{TV}:=\sup\left\{ \int_{\Ocal}u\div\eta\,d\xi\,\bigg\vert\,\eta\in C_{0}^{\infty}(\Ocal;\R^{d}),\,\|\eta\|_{L^{\infty}(\Ocal;\R^{d})}\le1\right\} 
\]
and let $BV$ be the space of functions of bounded variation, that
is, 
\[
BV(\Ocal):=\{u\in L^{1}(\Ocal)\;\vert\;\|u\|_{TV}<\infty\}.
\]

For a proper, convex, lower semi-continuous (l.s.c.) function $\Phi:H\to[0,+\infty]$,
we denote the subdifferential by $\partial\Phi$. The graph of $\partial\Phi$
consists precisely of the pairs of elements $(x,y)\in\partial\Phi\subseteq H\times H$
that satisfy $(y,z-x)_{H}\le\Phi(z)-\Phi(x)$ for all $z\in H$. In
this context, we may also write $y\in\partial\Phi(x)$, where we identify
the subdifferential as a multi-valued map $\partial\Phi:H\to2^{H}$.

We say that a function $X\in L^{1}([0,T]\times\Omega;H)$ is $\{\Fcal_{t}\}$-progressively
measurable if $X1_{[0,t]}$ is $\Bcal([0,t])\otimes\Fcal_{t}$-measurable
for all $t\in[0,T]$. The domain of (unbounded) linear operators $A$
is denoted by $\operatorname{dom}(A)$, and by the same notation,
we denote the effective domain of convex functionals or multi-valued
graphs. By $C$, we denote a positive constant that may change its
value from line to line.

\section{\label{sec:Hypo_commutation}Hypotheses and commutation relation}

Suppose that $\Ocal\subset\R^{d}$ is a sufficiently smooth, open,
bounded domain. Denote the surface element on $\partial\Ocal$ by
$S^{d-1}$. Denote by $\nu$ the outer unit normal on $\partial\Ocal$.

Below, we collect our assumptions on the ``diffusion matrix'' $\mathbf{b}$
and prove some essential properties of the associated partial differential
operators. Briefly summarized, we are assuming conditions to ensure
that
\begin{itemize}
\item the first-order partial differential operators associated to the rows
of $\mathbf{b}$ are well-defined unbounded skew-symmetric linear
operators on $L^{2}(\Ocal)$, see Assumption \ref{assu:well} (i)
and Lemma \ref{lem:B-properties} below;
\item the groups of diffeomorphisms generated by the rows of $\mathbf{b}$
mutually commute, see Assumption \ref{assu:well} (ii) and Lemma \ref{lem:lie-bracket}
below;
\item the partial differential operators associated to the rows of $\mathbf{b}$
leave the domain of the Neumann Laplace invariant and commute with
its resolvent, see Assumption \ref{assu:well} (i), (iii), (iv) and
Lemmas \ref{lem:Neumann-invariance} and \ref{lem:shigekawa} below.
\end{itemize}
We note that the commutation assumptions are typical for gradient-type
noise, even for linear stochastic equations, see \cite{DaPrato:1982wn,DaPrato:1982ci}
and \cite[Section 6.5]{DPZ92}.
\begin{assumption}
\label{assu:well} Suppose that the diffusion coefficients $b_{i}\in C^{2}(\overline{\Ocal};\R^{d})$,
$1\le i\le N$, satisfy the following conditions:
\begin{enumerate}
\item $\langle b_{i},\nu\rangle=0$ on $\partial\Ocal$ for all $1\le i\le N$.
\item Either $N=1$, or $b_{i}^{k}\partial_{k}b_{l}^{j}=b_{l}^{k}\partial_{k}b_{i}^{j}$
on $\overline{\Ocal}$ for all $1\le k,j\le d$, $1\le i,l\le N$,
$i\not=l$.
\item $\div b_{i}=0$ and $\langle\Delta b_{i},b_{i}\rangle=0$ on $\overline{\Ocal}$
for all $1\le i\le N$ (where the Laplace operator acts componentwise).
\item $\langle\langle\nabla b_{i},\nu\rangle,b_{i}\rangle+\langle\langle\nabla b_{i},b_{i}\rangle,\nu\rangle=0$
on $\partial\Ocal$ for all $1\le i\le N$ (where the gradient acts
componentwise).
\end{enumerate}
\end{assumption}

By Lemma \ref{lem:killing_lemma} in the appendix, sufficiently smooth
vector fields $b_{i}$ that satisfy (iii) and (iv) above, are precisely
the so-called \emph{Killing vector fields}, see \eqref{eq:kill_basic}
in the appendix for the definition.
\begin{rem}
Condition (i) and (iii) in the above Assumption sort out any examples
with stochastic perturbation for the case $d=1$. Indeed, let $\Ocal=I$
be a bounded interval, so that clearly $\nu=\pm1$ at the endpoints
of $I$. One the one hand, condition (i) implies that $b=0$ on $\partial I$.
On the other hand, $\div b=b^{\prime}=0$ implies that $b$ must be
constant on $\overline{I}$. Hence $b\equiv0$.

Altogether, condition (i) ensures that the noise coefficients respect
Neumann boundary conditions\footnote{If $b\equiv1$ on $\Ocal=I=(0,2\pi)$, then $\xi\mapsto\cos\xi$ has
Neumann boundary conditions, however, $b\cdot(\cos\xi)^{\prime}$
does not.}, see Lemma \ref{lem:Neumann-invariance} below.
\end{rem}

\begin{example}
\label{exa:2D} Let $N=1$, $d=2$. Let $\Ocal=\{\zeta\in\R^{2}\;\vert\;|\zeta|<R\}$,
$R>0$. Let $b(\xi):=(\xi_{2},-\xi_{1})$. Then \eqref{eq:main_equation}
becomes
\[
\left\{ \begin{aligned}dX_{t} & \in\div\left[\sgn\left(\nabla X_{t}\right)\right]dt+(\xi_{2}\partial_{1}X_{t}-\xi_{1}\partial_{2}X_{t})\,\circ d\beta_{t}, &  & \text{in\ensuremath{\;(0,T)\times\mathcal{{O}}},}\\
X_{0} & =x, &  & \text{in}\ensuremath{\;\mathcal{{O}}},\\
\dfrac{\partial X_{t}}{\partial\nu} & =0, &  & \text{on}\;(0,T)\times\mathcal{\partial{O}}.
\end{aligned}
\right.
\]

\end{example}

\begin{example}
\label{exa:cross-prod-1} Let $N=1$, $d=3$. Let $\Ocal=\{\zeta\in\R^{3}\;\vert\;|\zeta|<R\}$,
$R>0$. Let $b(\xi):=(\xi_{3}-\xi_{2},\xi_{1}-\xi_{3},\xi_{2}-\xi_{1})$
and denote $\mathbf{1}:=(1,1,1)$ (clearly, $b(\xi)=\xi\times\mathbf{1}$).
Then $b$ is a Killing vector field and \eqref{eq:main_equation}
becomes
\[
\left\{ \begin{aligned}dX_{t} & \in\div\left[\sgn\left(\nabla X_{t}\right)\right]dt+\langle\xi\times\nabla X_{t},\mathbf{1}\rangle\,\circ d\beta_{t}, &  & \text{in\ensuremath{\;(0,T)\times\mathcal{{O}}},}\\
X_{0} & =x, &  & \text{in}\ensuremath{\;\mathcal{{O}}},\\
\dfrac{\partial X_{t}}{\partial\nu} & =0, &  & \text{on}\;(0,T)\times\mathcal{\partial{O}}.
\end{aligned}
\right.
\]
One can replace $\mathbf{1}$ by any constant vector $\zeta_{0}\in\R^{3}\setminus\{0\}$
and get that $\tilde{b}(\xi):=\xi\times\zeta_{0}$ still satisfies
Assumption \ref{assu:well}.
\end{example}

\begin{rem}
Note that:
\begin{enumerate}
\item The above vector fields $\xi\mapsto(\xi_{2},-\xi_{1})$ and $\xi\mapsto\xi\times\zeta_{0}$
resp. are the infinitesimal generators of the rotation groups $SO(2)$
and $SO(3)$ resp., see e.g. \cite{hall2015lie}. They generate groups
of rotations around the origin, leaving balls centered at the origin
invariant, which explains why the respective domains are chosen as
above ($\zeta_{0}$ spans the axis of rotation).
\item Let $d=N$. The example of constant vector fields $b_{i}^{j}=\delta_{i,j}$
are precisely the infinitesimal generators of groups of translations
(violating Assumption \ref{assu:well} (i) on balls). The $d$-torus
$\mathbb{T}^{d}$ leaves the translation groups invariant, and is
still a bounded, convex domain, leading either to periodic boundary
conditions or to a setting for compact manifolds without boundary.
\end{enumerate}
\end{rem}
Recall that the domain $\operatorname{dom}(-\Delta)$ of the Neumann
Laplace in the weak sense is given by all elements $u\in H^{1}(\Ocal)$
such that $\Delta u\in L^{2}(\Ocal)$ and such that
\[
\int_{\Ocal}v\Delta u\,d\xi=-\int_{\Ocal}\langle\nabla v,\nabla u\rangle\,d\xi\quad\forall v\in H^{1}(\Ocal).
\]
For $u\in\operatorname{dom}(-\Delta)$, the normal derivative $\frac{\partial u}{\partial\nu}$
belongs to $H^{-1/2}(\partial\Ocal)$ (being the dual of the\emph{
}space of traces $H^{1/2}(\partial\Ocal)$) and is zero, see e.g.
\cite[p. 250]{Demengel:2012dp} for details. As we assume smooth boundary,
the normal derivative is given by $\frac{\partial u}{\partial\nu}=\langle\nabla u,\nu\rangle$
$S^{d-1}$-a.e., whenever $u\in C^{2}(\overline{\Ocal})$. Hence,
\begin{equation}
\Ccal:=\{u\in C^{2}(\overline{\Ocal})\;\vert\;\langle\nabla u,\nu\rangle=0\;S^{d-1}\text{-a.e.\ensuremath{\}}}\label{eq:core_defi}
\end{equation}
 is a \emph{core} for the Neumann Laplace, that is, dense in $\operatorname{dom}(-\Delta)$
w.r.t. to the graph norm
\[
\|u\|_{\operatorname{dom}(-\Delta)}^{2}:=\int_{\Ocal}(|\Delta u|^{2}+|u|^{2})\,d\xi.
\]
On the domain $H^{1}(\Ocal)$, we define the linear operators $B_{i}$,
$1\le i\le N$ as 
\[
B_{i}:H^{1}(\mathcal{O})\rightarrow L^{2}(\mathcal{O})
\]
\begin{eqnarray*}
B_{i}(u)(\xi) & := & \langle b_{i}(\xi),\nabla u(\xi)\rangle\\
 & = & \div[b_{i}(\xi)u(\xi)],\quad\forall u\in H^{1}\left(\mathcal{O}\right),
\end{eqnarray*}
where $b_{i}$ satisfies Assumption \ref{assu:well}.
\begin{lem}
\label{lem:B-properties} Assume Assumption \ref{assu:well}. Let
us collect the following properties:
\begin{enumerate}
\item The space $H^{1}(\Ocal)$ is the domain of skew-adjointness of $B_{i}$,
$1\le i\le N$, that is
\[
B_{i}u=-B_{i}^{\ast}u,\quad1\le i\le N,\quad u\in H^{1}(\Ocal),
\]
where $B_{i}^{\ast}$ denotes the adjoint operator in $L^{2}(\Ocal)$.
\item For all $u\in H^{1}(\Ocal)$, it holds that
\begin{equation}
\int_{\Ocal}uB_{i}u\,d\xi=0,\quad1\le i\le N.\label{eq:uBu}
\end{equation}

\end{enumerate}
\end{lem}

\begin{proof}
Let $u\in H^{1}(\Ocal)$ and fix $1\le i\le N$.
\begin{enumerate}
\item [(i):] By the Gauss-Green theorem, for $v\in H^{1}(\Ocal)$, taking
Assumption \ref{assu:well} (i) into account,
\[
\begin{aligned}(B_{i}^{\ast}u,v)_{L^{2}(\Ocal)} & =(u,B_{i}v)_{L^{2}(\Ocal)}\\
 & =\int_{\Ocal}u\langle b_{i},\nabla v\rangle\,d\xi\\
 & =-\int_{\Ocal}\div(b_{i}u)v\,d\xi+\int_{\partial\Ocal}uv\langle b_{i},\nu\rangle\,dS^{d-1}\\
 & =-\int_{\Ocal}\langle b_{i},\nabla u\rangle v\,d\xi\\
 & =-(B_{i}u,v)_{L^{2}(\Ocal)}.
\end{aligned}
\]
The density of $H^{1}(\Ocal)\subset L^{2}(\Ocal)$ yields (i).
\item [(ii):] This follows directly from (i).
\end{enumerate}
\end{proof}

For $1\le i\le N$ fixed, let $e^{tB_{i}}:L^{2}(\Ocal)\to L^{2}(\Ocal)$,
$t\in\R$, denote the $C_{0}$-group of linear operators associated
to $B_{i}$, such that, in particular,
\[
\frac{d}{dt}e^{tB_{i}}u\bigg\vert_{t=0}=B_{i}u,\quad u\in H^{1}(\Ocal).
\]

\begin{lem}
\label{lem:lie-bracket} The groups $e^{tB_{i}}$, $1\le i\le N$,
$t\in\R$ mutually commute, whenever Assumption \ref{assu:well} holds.\end{lem}
\begin{proof}
For $N=1$, there is nothing to prove. Let $b_{i}$, $b_{l}$, $i\not=l$
be as above. Let $f\in C^{2}(\overline{\Ocal})$. Define the commutator
$[b_{i},b_{l}]f:=B_{i}B_{l}f-B_{l}B_{i}f$. By Leibniz's rule,
\[
\begin{aligned}\left[b_{i},b_{l}\right]f & =\sum_{1\le k,j\le d}b_{i}^{k}\partial_{k}(b_{l}^{j}\partial_{j}f)-b_{l}^{k}\partial_{k}(b_{i}^{j}\partial_{j}f)\\
 & =\sum_{1\le k,j\le d}b_{i}^{k}b_{l}^{j}\partial_{k}\partial_{j}f+b_{i}^{k}\partial_{k}b_{l}^{j}\partial_{j}f-b_{l}^{k}b_{i}^{j}\partial_{k}\partial_{j}f-b_{l}^{k}\partial_{k}b_{i}^{j}\partial_{j}f\\
 & =\sum_{1\le k,j\le d}(b_{i}^{k}\partial_{k}b_{l}^{j}-b_{l}^{k}\partial_{k}b_{i}^{j})\partial_{j}f\\
 & =0,
\end{aligned}
\]
where we have used Assumption \ref{assu:well} (ii) in the last step.

Now, for any $1\le i\le N$, denote by $Z_{t}^{i}:\overline{\Ocal}\to\overline{\Ocal}$,
$t\in[0,\infty)$ the flow of diffeomorphisms on $\overline{\Ocal}$
corresponding to the vector field action of $b_{i}$, that is,
\[
\frac{d}{dt}Z_{t}^{i}=b_{i}(Z_{t}^{i}),\quad t\ge0,\quad Z_{0}^{i}=\xi\in\overline{\Ocal}.
\]
Then, since we have proved above that $[b_{i},b_{l}]=0$ for any $i\not=l$,
and from the property of the vector fields to be divergence free,
we get that the $Z^{i}$, $1\le i\le N$ mutually commute (in the
sense of composition of maps), see \cite[Ch. I.2, Exercise 3]{sakai1996riemannian}.
However, it holds that $(e^{tB_{i}}u)(\xi)=u(Z_{t}^{i}(\xi))$ and
one easily deduces that the groups of operators commute on $L^{2}$.
\end{proof}

\begin{lem}
\label{lem:Neumann-invariance} Suppose that Assumption \ref{assu:well}
holds. Let $u\in C^{3}(\overline{\mathcal{O}})$ be a scalar function
with $\langle\nabla u,\nu\rangle=0$ on $\partial\mathcal{O}$. Then
it holds that $\langle\nabla(\langle b_{i},\nabla u\rangle),\nu\rangle=0$
on $\partial\mathcal{O}$ for every $1\le i\le N$.\end{lem}
\begin{proof}
Fix $1\le i\le N$ and set $b:=b_{i}$. Let $\eta\in C^{2}(\mathcal{\overline{O}})$
be a test-function. We claim that for any $u\in C^{3}(\overline{\mathcal{O}})$
with $\langle\nabla u,\nu\rangle=0$ on $\partial\mathcal{O}$, it
holds that 
\begin{equation}
\int_{\partial\mathcal{O}}\eta\langle\nabla(\langle b,\nabla u\rangle),\nu\rangle\,dS^{d-1}=0\quad\forall\eta\in C^{2}(\overline{\mathcal{O}}).\label{eq:boundaryzero}
\end{equation}
In order to prove \eqref{eq:boundaryzero}, we first apply Gauss's
divergence theorem to the vector field $F:=\eta\nabla(\langle b,\nabla u\rangle),$
and get that
\[
\int_{\partial\mathcal{O}}\langle\eta\nabla(\langle b,\nabla u\rangle),\nu\rangle\,dS^{d-1}=\int_{\mathcal{O}}\div F\,d\xi.
\]
However, $\text{div \ensuremath{F}}=\langle\nabla\eta,\nabla(\langle b,\nabla u\rangle)\rangle+\eta\Delta(\langle b,\nabla u\rangle)$.
Let us begin with investigating the second term. By the Killing assumption,
we have the commutation on sufficiently smooth functions (cf. Theorem
\ref{thm:commu} in the appendix), thus, $\eta\Delta(\langle b,\nabla u\rangle)=\eta\langle b,\nabla\Delta u\rangle=\langle\eta b,\nabla\Delta u\rangle$.
Integrating by parts, we get that
\[
\int_{\mathcal{O}}\langle\eta b,\nabla\Delta u\rangle\,d\xi=-\int_{\mathcal{O}}\div(\eta b)\Delta u\,d\xi+\int_{\partial\mathcal{O}}\eta\langle b,\nu\rangle\Delta u\,dS^{d-1}.
\]
The latter term is zero by $\langle b,\nu\rangle=0$. Also, since
both $\eta$ and $b$ are smooth up to the boundary, $\div(\eta b)\in H^{1}(\mathcal{O})$.
So we can use the Neumann boundary condition for $u$ to get that
\[
-\int_{\mathcal{O}}\div(\eta b)\Delta u\,d\xi=\int_{\mathcal{O}}\langle\nabla\div(\eta b),\nabla u\rangle\,d\xi.
\]
Clearly, as $\div b=0$ , we get that $\nabla\div(\eta b)=\nabla(\langle b,\nabla\eta\rangle)$.
Hence
\[
\int_{\mathcal{O}}\div F\,d\xi=\int_{\mathcal{\mathcal{O}}}[\langle\nabla\eta,\nabla(\langle b,\nabla u\rangle)\rangle+\langle\nabla(\langle b,\nabla\eta\rangle),\nabla u\rangle]\,d\xi,
\]
differentiating out this term yields
\[
\int_{\mathcal{O}}\div F\,d\xi=\int_{\mathcal{O}}[\langle(Db)\cdot\nabla u,\nabla\eta\rangle+\langle(D^{2}u)\cdot b,\nabla\eta\rangle+\langle(Db)\cdot\nabla\eta,\nabla u\rangle+\langle(D^{2}\eta)\cdot b,\nabla u\rangle]\,d\xi,
\]
where, $Db$ denotes the Jacobian of $b$ and $D^{2}$ denotes the
Hessian of a scalar function, ``$\cdot$'' denotes matrix multiplication.

However, $Db$ is skew-symmetric with respect to the Euclidean scalar
product due to the Killing assumption, see \eqref{eq:kill_basic}
in the appendix. Hence
\[
\langle(Db)\cdot\nabla u,\nabla\eta\rangle=-\langle(Db)\cdot\nabla\eta,\nabla u\rangle
\]
and the above term becomes
\[
\int_{\mathcal{O}}\div F\,d\xi=\int_{\mathcal{O}}[\langle(D^{2}u)\cdot b,\nabla\eta\rangle+\langle(D^{2}\eta)\cdot b,\nabla u\rangle]\,d\xi.
\]
With Einstein's summation convention, interchanging the order of differentiation,
\[
\int_{\mathcal{O}}[\partial_{i}\partial_{j}ub^{j}\partial_{i}\eta+\partial_{i}\partial_{j}\eta b^{j}\partial_{i}u]\,d\xi=\int_{\mathcal{O}}[\partial_{j}\partial_{i}ub^{j}\partial_{i}\eta+\partial_{i}\partial_{j}\eta b^{j}\partial_{i}u]\,d\xi.
\]
Integrating by parts in the first term yields
\[
\int_{\mathcal{O}}\partial_{j}\partial_{i}ub^{j}\partial_{i}\eta\,d\xi=-\int_{\mathcal{O}}\partial_{i}u\partial_{j}(b^{j}\partial_{i}\eta)\,d\xi+\int_{\partial\mathcal{O}}\partial_{i}ub^{j}\partial_{i}\eta\nu_{j}\,dS^{d-1}.
\]
Now, in the boundary integral term, we can separate the sums over
$j$ and $i$ resp. and get that this term becomes zero by $\langle b,\nu\rangle=0$.
Furthermore,
\[
-\partial_{i}u\partial_{j}(b^{j}\partial_{i}\eta)=-\partial_{i}u\partial_{j}b^{j}\partial_{i}\eta-\partial_{i}ub^{j}\partial_{j}\partial_{i}\eta=-\partial_{i}ub^{j}\partial_{j}\partial_{i}\eta,
\]
as we have that $\text{div \ensuremath{b}}=0$. Finally, the remaining
terms cancel, and we get that 
\[
0=\int_{\mathcal{O}}\div F\,d\xi=\int_{\partial\mathcal{O}}\eta\langle\nabla(\langle b,\nabla u\rangle),\nu\rangle\,dS^{d-1}\quad\forall\eta\in C^{2}(\overline{\mathcal{O}}).
\]

\end{proof}
We shall need the following commutation result. Denote the resolvent
of the Neumann Laplace by $J_{\delta}:=\left(\operatorname{Id}-\delta\Delta\right)^{-1}$,
$\delta>0$.
\begin{thm}[Shigekawa]
\label{thm:shigekawa}  Fix $1\le i\le N$. Suppose that there exists
a linear subspace $\Dcal\subset\operatorname{dom}(-\Delta)$ such
that the following conditions hold:
\begin{enumerate}
\item $\Delta(\Dcal)\subseteq\operatorname{dom}(B_{i})$,
\item $B_{i}(\Dcal)\subseteq\operatorname{dom}(-\Delta)$,
\item $\Dcal$ is a core (see \eqref{eq:core_defi} for the terminology)
for $(-\Delta,\operatorname{dom}(-\Delta))$,
\item $\operatorname{dom}(-\Delta)\subseteq\operatorname{dom}(B_{i})$ and
$\operatorname{dom}(-\Delta)\subseteq\operatorname{dom}(B_{i}^{\ast})$,
\item For any $u\in\Dcal$, it holds that
\[
B_{i}\Delta u=\Delta B_{i}u.
\]

\end{enumerate}
Then for all $\delta>0$, and every $u\in\operatorname{dom}(B_{i})$,
it holds that

\[
B_{i}J_{\delta}u=J_{\delta}B_{i}u.
\]
\end{thm}
\begin{proof}
See \cite[Theorem 3.1 and Proposition 3.2]{Shigekawa:2000io}.
\end{proof}

\begin{lem}
\label{lem:shigekawa} Assumption \ref{assu:well} implies that all
of the conditions of Theorem \ref{thm:shigekawa} are satisfied for
all $1\le i\le N$.\end{lem}
\begin{proof}
Assume the conditions of Assumption \ref{assu:well}. Fix $1\le i\le N$.
First note that Lemma \ref{lem:B-properties} implies that $\operatorname{dom}(B_{i})=\operatorname{dom}(B_{i}^{\ast})=H^{1}(\Ocal)$.
Let $\Dcal:=C^{\infty}(\overline{\Ocal})\cap\Ccal$, where $\Ccal:=\{u\in C^{2}(\overline{\Ocal})\;\vert\;\langle\nabla u,\nu\rangle=0\;S^{d-1}\text{-a.e.\ensuremath{\}}}$.
Obviously, $\Dcal$ is a core for $\operatorname{dom}(-\Delta)$.
Hence (i), (iii), (iv) are clearly satisfied. (ii) follows from Lemma
\ref{lem:Neumann-invariance}. The commutation on smooth functions
(v) follows from Theorem \ref{thm:commu}, since $b_{i}$ is a Killing
field by Assumption \ref{assu:well} and Lemma \ref{lem:killing_lemma}.
\end{proof}
Let us also define $B_{i}^{2}:H^{1}(\Ocal)\to(H^{1}(\Ocal))^{\ast}$
by
\[
B_{i}^{2}u:=-B_{i}^{\ast}B_{i}u,\quad u\in H^{1}(\Ocal),\quad1\le i\le N.
\]
In the sense of Schwartz distributions, it holds that
\[
\sum_{i=1}^{N}B_{i}^{2}u=\div[\mathbf{b}^{\ast}\mathbf{b}\nabla u].
\]
Set $S:=H^{1}(\Ocal)$. We thus have a Gelfand triple
\[
S\subset H\subset S^{\ast}.
\]

\section{\label{sec:SVI}Stochastic variational inequalities (SVI)}

Let $\beta=(\beta^{1},\ldots,\beta^{N})$ be a $N$-dimensional Brownian
motion on a filtered (normal) probability space $\left(\Omega,\mathcal{F},\left\{ \mathcal{F}_{t}\right\} _{t\ge0},\mathbb{P}\right)$
(with expected value $\E[Y]=\int_{\Omega}Y\,d\P$, $Y\in L^{1}(\Omega)$).
We consider the following SPDE on $H=L^{2}(\Ocal)$, where $\Ocal\subset\R^{d}$
is a smooth, open, bounded domain such that $\Ocal$ or $\partial\Ocal$
is convex, 
\begin{equation}
\left\{ \begin{aligned}dX_{t} & \in\div\left[\Psi\left(\nabla X_{t}\right)\right]dt+\dfrac{1}{2}\sum_{i=1}^{N}B_{i}^{2}X_{t}\,dt+\sum_{i=1}^{N}B_{i}X_{t}\,d\beta_{t}^{i}, &  & \text{in\ensuremath{\;(0,T)\times\mathcal{{O}},}}\\
X_{0} & =x, &  & \text{in}\ensuremath{\;\mathcal{{O}}},\\
\dfrac{\partial X_{t}}{\partial\nu} & =0, &  & \text{on}\;(0,T)\times\mathcal{\partial{O}},
\end{aligned}
\right.\label{eq:main_equation-2}
\end{equation}
here, $\Psi:=\partial\varphi\subseteq\R^{d}\times\R^{d}$ is the subdifferential
of $\varphi:=\frac{{1}}{p}|\cdot|^{p}$ for $p\in[1,2)$, which is
multi-valued for $p=1$, i.e. $\partial(\xi\mapsto|\xi|)(\cdot)=\sgn(\cdot)$.
More precisely, after fixing $p\in[1,2)$, let 
\[
\tilde{\Phi}(u):=\begin{cases}
\int_{\Ocal}\varphi(\nabla u(\xi))\,d\xi, & \text{if }u\in H^{1},\\
+\infty, & \text{if \ensuremath{u\in L^{2}\setminus H^{1}.}}
\end{cases}
\]
$\tilde{\Phi}$ is a proper convex functional on $L^{2}$ but might
fail to be lower semi-continuous. Let us define 
\[
\Phi(u):=\operatorname{cl}\tilde{{\Phi}}(u):=\inf\left\{ \liminf_{n\to\infty}~\tilde{{\Phi}}(u_{n})\;\big\vert~u_{n}\rightarrow u\in L^{2}\left(\mathcal{O}\right)\;\text{{strongly}}\right\} ,
\]
the so-called \emph{lower semi-continuous envelope} of $\tilde{{\Phi}}$,
cf. \cite[Proposition 11.1.1]{ABM06}. The l.s.c. envelope is given
by, for $p\in(1,2)$, 
\[
\Phi(u):=\begin{cases}
\int_{\Ocal}\varphi(\nabla u)\,d\xi & \text{if }u\in W^{1,p}(\Ocal)\cap L^{2}(\Ocal),\\
+\infty & \text{if }u\in L^{2}(\Ocal)\setminus W^{1,p}(\Ocal),
\end{cases}
\]
and for $p=1$, 
\[
\Phi(u):=\begin{cases}
\|u\|_{TV} & \text{if }u\in BV(\Ocal)\cap L^{2}(\Ocal),\\
+\infty & \text{if }u\in L^{2}(\Ocal)\setminus BV(\Ocal).,
\end{cases}
\]
where we suppress the dependence on $p$ in the notation. Obviously,
$\Phi$ is convex and it is easy to see that $\Phi$ is lower semi-continuous
on $H$. Moreover, $\tilde{{\Phi}}$ is Gâteaux-differentiable in
$u$ with derivative given by 
\[
D\tilde{{\Phi}}(u)(v)=\int_{\Ocal}\langle\eta,\nabla v\rangle\,d\xi,
\]
with $\eta(\xi)\in\Psi(\nabla u(\xi))$ for a.e. $\xi\in\Ocal$. In
fact, $\Phi$ coincides with the lower semi-continuous hull of $\tilde{{\Phi}}$
on $H$, and we have for $u\in H^{1}$ that 
\[
\{-\div\eta\;\vert\;\,\eta\in H^{1}(\Ocal;\R^{d}),\,\eta\in\Psi(\nabla u),\,d\xi\text{-a.e.}\}\subseteq\partial\Phi(u).
\]
However, the full characterization of $\partial\Phi$ (already in
the space $L^{1}(\Ocal)$) is involved. We shall omit its precise
characterization and instead refer to \cite{ABM06}.

Equation \eqref{eq:main_equation-2} is then written in relaxed form
as 
\begin{equation}
\left\{ \begin{aligned}dX_{t} & \in-\partial\Phi(X_{t})\,dt+\dfrac{1}{2}\sum_{i=1}^{N}B_{i}^{2}X_{t}\,dt+\sum_{i=1}^{N}B_{i}X_{t}\,d\beta_{t}^{i}, &  & \text{\ensuremath{\;t\in(0,T),}}\\
X_{0} & =x. &  & \text{}
\end{aligned}
\right.\label{eq:abstract-equation}
\end{equation}
Motivated by \cite{singular1,VBMR}, let us define our notion of a
solution to \eqref{eq:abstract-equation}.

\begin{defn}
\label{def:SVI}Let $x\in L^{2}(\Omega,\Fcal_{0},\P;L^{2}(\Ocal))$,
$T>0$. An $\{\Fcal_{t}\}$-progressively measurable process $X\in L^{2}([0,T]\times\Omega;L^{2}(\Ocal))$
is called an \emph{SVI-solution} to \eqref{eq:abstract-equation}
if 
\begin{enumerate}
\item \emph{(Regularity)} 
\begin{equation}
\Phi(X)\in L^{1}([0,T]\times\Omega).\label{eq:reg}
\end{equation}

\item \emph{(Variational inequality)} For every $Z\in L^{2}([0,T]\times\Omega;H^{1}(\Ocal))$
such that there exist $Z_{0}\in L^{2}(\Omega,\Fcal_{0},\P;H^{1}(\Ocal))$,
$G\in L^{2}([0,T]\times\Omega;L^{2}(\Ocal))$, $\{\Fcal_{t}\}$-progressively
measurable, such that the following equality holds $L^{2}(\Ocal)$,
that is, 
\begin{equation}
Z_{t}=Z_{0}+\int_{0}^{t}G_{s}\,ds+\frac{{1}}{2}\sum_{i=1}^{N}\int_{0}^{t}B_{i}^{2}Z_{s}\,ds+\sum_{i=1}^{N}\int_{0}^{t}B_{i}Z_{s}\,d\beta_{s}^{i},\label{eq:test-eq}
\end{equation}
$\P$-a.s. for all $t\in[0,T]$, we have that the following variational
inequality holds true 
\begin{equation}
\begin{split}\frac{{1}}{2}\E\|X_{t}-Z_{t}\|_{L^{2}(\Ocal)}^{2} & +\E\int_{0}^{t}\Phi(X_{s})\,ds\\
 & \le\frac{{1}}{2}\E\|x-Z_{0}\|_{L^{2}(\Ocal)}^{2}+\E\int_{0}^{t}\Phi(Z_{s})\,ds\\
 & \qquad-\E\int_{0}^{t}(G_{s},X_{s}-Z_{s})_{L^{2}(\Ocal)}\,ds,
\end{split}
\label{eq:SVI}
\end{equation}
for almost all $t\in[0,T]$. 
\end{enumerate}
\end{defn}
Moreover, if $X\in L^{2}(\Omega;C([0,T];L^{2}(\Ocal)))$, we say that
$X$ is a \emph{(time-) continuous SVI solution} to \eqref{eq:abstract-equation}.

\begin{rem}
\label{rem:test-func-rem}Practically, the test-process $Z$ needs
to satisfy $Z\in L^{2}([0,T]\times\Omega;H^{2}(\Ocal))$, we shall
provide in \eqref{app1-1} below that a process $Z$ of the form \eqref{eq:test-eq}
in fact exists (see also \eqref{H2} below).
\end{rem}
Inequality \eqref{eq:SVI} is obtained by formally applying the Itô
formula for the square of the $H$-norm to the process 
\[
d(X-Z)=(-\partial\Phi(X)-G)\,dt+\frac{{1}}{2}\sum_{i=1}^{N}B_{i}^{2}(X-Z)\,dt+\sum_{i=1}^{N}B_{i}(X-Z)\,d\beta^{i},
\]
taking expectation and using the subdifferential property.

\section{\label{sec:Existence-and-uniqueness}Existence and uniqueness}

\subsection{Existence}
\begin{thm}
\label{thm:main_thm}Let $x\in L^{2}(\Omega,\Fcal_{0},\P;H)$. Then
there is a unique continuous SVI solution $X\in L^{2}(\Omega;C([0,T];H))$
to \eqref{eq:abstract-equation} in the sense of Definition \ref{def:SVI}.
For two SVI solutions $X,Y$ with initial conditions $x,y\in L^{2}(\Omega,\Fcal_{0},\P;H)$,
resp., we have 
\[
\operatorname{ess\,sup}\displaylimits_{t\in[0,T]}\E\|X_{t}-Y_{t}\|_{H}^{2}\le\E\|x-y\|_{H}^{2}.
\]

\end{thm}

\begin{proof}
Recall the notation $H=L^{2}$, $S=H^{1}$. We first assume an initial
condition $x_{0}\in L^{2}(\Omega,\Fcal_{0},\P;S)$ and, in the last
part of the proof, we shall generalize to $x_{0}\in L^{2}(\Omega,\Fcal_{0},\P;H)$.

In order to prove the existence of the solution, we need to take a
threefold approximation for equation \eqref{eq:main_equation-2}.
Therefore, we consider the following regularized equation,

\begin{equation}
\left\{ \begin{array}{ll}
dX_{t}^{\varepsilon,\lambda,\delta}=J_{\delta}\operatorname{div}\Psi^{\lambda}\left(\nabla J_{\delta}X_{t}^{\varepsilon,\lambda,\delta}\right)dt+\varepsilon\Delta X_{t}^{\varepsilon,\lambda,\delta}dt & ~\text{in}~\left(0,T\right)\times\mathcal{O},\\
\quad\quad\quad+\dfrac{1}{2}\sum\limits _{i=1}^{N}(B_{i}^{\delta})^{2}(X_{t}^{\varepsilon,\lambda,\delta})dt+\sum\limits _{i=1}^{N}B_{i}^{\delta}\left(X_{t}^{\varepsilon,\lambda,\delta}\right)d\beta_{t}^{i},\\
\medskip X_{0}^{\varepsilon,\lambda,\delta}=x_{0}, & ~\text{in}~\mathcal{O},\\
\dfrac{\partial X_{t}^{\varepsilon,\lambda,\delta}}{\partial\nu}=0, & ~\text{on}~\left(0,T\right)\times\partial\mathcal{O},
\end{array}\right.\label{app1}
\end{equation}
where $\Psi^{\lambda}$, $\lambda>0$, is the Yosida approximation
of $\Psi$ (cf. \cite[p. 37]{nonlin}), $J_{\delta}$, $\delta>0$,
is the resolvent of the Neumann Laplacian $L:=-\Delta$, i.e., $J_{\delta}=\left(\operatorname{Id}-\delta\Delta\right)^{-1}$
and $B_{i}^{\delta}(\cdot)=B_{i}(J_{\delta}(\cdot))$.

By \cite[Theorem 2.4]{BBHT}, we have that the there exists a unique
$\{\mathcal{F}_{t}\}$-adapted solution with $X\in C([0,T];H)\cap L^{2}([0,T];S)$
$\mathbb{P}$-a.s. such that \eqref{app1} holds $\mathbb{P}$-a.s.
in $L^{2}([0,T];S^{\ast})$. We note that our hypotheses guarantee
that the conditions needed for \cite[Theorem 2.4]{BBHT} are satisfied.

\emph{Step I (the estimate in $H^{2}\left(\mathcal{O}\right)$):}

Considering $J_{\alpha}$, $\alpha>0$, the resolvent of the Neumann
Laplace operator $-\Delta$, we define the sequence of semi-inner
products on $H$ 
\[
\left(u,v\right)_{\alpha}:=\left((-\Delta)_{\alpha}u,v\right)_{H},\quad u,v\in H,
\]
where $(-\Delta)_{\alpha}$ is the Yosida approximation of the operator
$-\Delta$, i.e., $(-\Delta)_{\alpha}=\frac{{1}}{\alpha}(\operatorname{Id}-J_{\alpha})=-\Delta J_{\alpha}$
and the induced semi-norms 
\[
\left\Vert u\right\Vert _{\alpha}:=\left\Vert (-\Delta)_{\alpha}^{\frac{1}{2}}u\right\Vert _{H},\quad u\in H,
\]
where $(-\Delta)_{\alpha}^{\frac{1}{2}}$ denotes the operator square
root.

Since they are continuous on $H$ and for all $u\in S$ we have that
\[
\left\Vert u\right\Vert _{\alpha}\longrightarrow\left\Vert \nabla u\right\Vert _{L^{2}(\Ocal;\R^{d})}\quad\text{as}~\alpha\rightarrow0.
\]

We shall apply the Itô formula \cite[Theorem 4.2.5]{concise} to \eqref{app1}
with the functional $u\mapsto\Vert u\Vert_{\alpha}^{2}$, for $\varepsilon,~\lambda$
and $\delta$ fixed, and we get that for all $t\in[0,T]$ and $\P$-a.s.
\begin{eqnarray}
\left\Vert X_{t}^{\varepsilon,\lambda,\delta}\right\Vert _{\alpha}^{2} & = & \left\Vert x_{0}\right\Vert _{\alpha}^{2}+2\int_{0}^{t}\left((-\Delta)_{\alpha}X_{s}^{\varepsilon,\lambda,\delta},J_{\delta}\operatorname{div}\Psi^{\lambda}\left(\nabla J_{\delta}X_{s}^{\varepsilon,\lambda,\delta}\right)\right)_{H}ds\label{ito1}\\
 &  & +2\varepsilon\int_{0}^{t}\phantom{\!}_{S}\langle(-\Delta)_{\alpha}X_{s}^{\varepsilon,\lambda,\delta},\Delta X_{s}^{\varepsilon,\lambda,\delta}\rangle_{S^{\ast}}ds\notag\\
 &  & +\sum_{i=1}^{N}\int_{0}^{t}\left((-\Delta)_{\alpha}X_{s}^{\varepsilon,\lambda,\delta},(B_{i}^{\delta})^{2}X_{s}^{\varepsilon,\lambda,\delta}\right)_{H}ds\notag\\
 &  & +2\sum_{i=1}^{N}\int_{0}^{t}\left((-\Delta)_{\alpha}X_{s}^{\varepsilon,\lambda,\delta},B_{i}^{\delta}X_{s}^{\varepsilon,\lambda,\delta}d\beta_{s}^{i}\right)_{H}\notag\\
 &  & +\sum_{i=1}^{N}\int_{0}^{t}\left\Vert (-\Delta)_{\alpha}^{\frac{1}{2}}B_{i}^{\delta}X_{s}^{\varepsilon,\lambda,\delta}\right\Vert _{H}^{2}ds.\notag
\end{eqnarray}
By well-known properties of the resolvent (as symmetry in $L^{2}$,
commutation with the Yosida approximation) and keeping in mind that
the the operators $B_{i}$ commute with the resolvent of the Neumann
Laplace by Theorem \ref{thm:shigekawa} and Lemma \ref{lem:shigekawa},
we can easily see that, by setting,
\[
\Phi_{\lambda}(u):=\int_{\Ocal}\varphi^{\lambda}(\nabla u)\,d\xi,\quad u\in S,
\]
and setting $v=X_{s}^{\varepsilon,\lambda,\delta}$, that 
\[
\left((-\Delta)_{\alpha}v,J_{\delta}\operatorname{div}\Psi^{\lambda}\left(\nabla J_{\delta}v\right)\right)_{H}=-\frac{{1}}{\alpha}(v-J_{\alpha}v,\partial(\Phi_{\lambda}\circ J_{\delta})v)_{H},
\]
cf. \cite[Proposition II.7.8]{Show} for the chain rule for subdifferentials.
By using the argument of \cite[Equation (3.7)]{GessToelle15}\footnote{For Dirichlet boundary conditions on piecewise convex domains, this
result has been proved directly without the use of heat kernel estimates
in in \cite[Appendix]{VBMR}.} (here, the convexity assumption on the boundary is needed, see also
\cite[Example 7.11]{GessToelle14}, where the heat kernel estimates
of \cite{Wang:2009bc,WangYan13} are applied), we see that
\[
\begin{aligned} & \frac{{1}}{\alpha}(J_{\alpha}v-v,\partial(\Phi_{\lambda}\circ J_{\delta})v)_{H}\\
\le & \frac{1}{\alpha}\left(\Phi_{\lambda}(J_{\delta}J_{\alpha}v)-\Phi_{\lambda}(J_{\delta}v)\right)=\frac{1}{\alpha}\left(\Phi_{\lambda}(J_{\alpha}J_{\delta}v)-\Phi_{\lambda}(J_{\delta}v)\right)\le0.
\end{aligned}
\]
Note that, since $v\in H^{1}(\Ocal)$, we have that 
\[
\phantom{\!}_{S}\langle(-\Delta)_{\alpha}v,\Delta v\rangle_{S^{\ast}}\le-\left\Vert (-\Delta)_{\alpha}v\right\Vert _{H}^{2}.
\]
To see this, just take into account that
\[
0\le\frac{1}{\alpha}(\nabla v-\nabla J_{\alpha}v,\nabla v-\nabla J_{\alpha}v)_{L^{2}(\Ocal;\R^{d})}=((-\Delta)_{\alpha}v,\Delta J_{\alpha}v)_{H}-\phantom{\!}_{S}\langle(-\Delta)_{\alpha}v,\Delta v\rangle_{S^{\ast}}.
\]
Furthermore, by commutation (see Theorem \ref{thm:shigekawa} and
Lemma \ref{lem:shigekawa}),
\[
\begin{aligned} & \left((-\Delta)_{\alpha}v,(B_{i}^{\delta})^{2}v\right)_{H}\\
= & -\left((-\Delta)_{\alpha}v,J_{\delta}B_{i}^{\ast}B_{i}^{\delta}v\right)_{H}=-\left(B_{i}^{\delta}(-\Delta)_{\alpha}v,B_{i}^{\delta}v\right)_{H}\\
= & -\left(B_{i}^{\delta}(-\Delta)_{\alpha}^{\frac{1}{2}}v,B_{i}^{\delta}(-\Delta)_{\alpha}^{\frac{1}{2}}v\right)_{H}=-\left\Vert (-\Delta)_{\alpha}^{\frac{1}{2}}B_{i}^{\delta}v\right\Vert _{H}^{2}.
\end{aligned}
\]
By going back and replacing in \eqref{ito1} we get that $\P\otimes ds$-a.s.,
\begin{eqnarray}
\left\Vert X_{t}^{\varepsilon,\lambda,\delta}\right\Vert _{\alpha}^{2} & \leq & \left\Vert x_{0}\right\Vert _{\alpha}^{2}-2\varepsilon\int_{0}^{t}\left\Vert (-\Delta)_{\alpha}X_{s}^{\varepsilon,\lambda,\delta}\right\Vert _{H}^{2}ds\label{ito2}\\
 &  & -\sum_{i=1}^{N}\int_{0}^{t}\left\Vert \left(-\Delta\right)_{\alpha}^{\frac{1}{2}}B_{i}^{\delta}\left(X_{s}^{\varepsilon,\lambda,\delta}\right)\right\Vert _{H}^{2}ds\nonumber \\
 &  & +2\sum_{i=1}^{N}\int_{0}^{t}\left(\left(-\Delta\right)_{\alpha}X_{s}^{\varepsilon,\lambda,\delta},B_{i}^{\delta}X_{s}^{\varepsilon,\lambda,\delta}d\beta_{s}^{i}\right)_{H}\nonumber \\
 &  & +\sum_{i=1}^{N}\int_{0}^{t}\left\Vert \left(-\Delta\right)_{\alpha}^{\frac{1}{2}}B_{i}^{\delta}\left(X_{s}^{\varepsilon,\lambda,\delta}\right)\right\Vert _{H}^{2}ds.\nonumber 
\end{eqnarray}
Taking the expectation and letting $\alpha\rightarrow0$ yields 
\begin{equation}
\mathbb{E}\left\Vert \nabla X_{t}^{\varepsilon,\lambda,\delta}\right\Vert _{L^{2}(\Ocal;\R^{d})}^{2}+2\varepsilon\mathbb{E}\int_{0}^{t}\left\Vert \Delta X_{s}^{\varepsilon,\lambda,\delta}\right\Vert _{H}^{2}ds\leq\E\left\Vert \nabla x_{0}\right\Vert _{L^{2}(\Ocal;\R^{d})}^{2}.\label{H2}
\end{equation}

\emph{Step II ($\delta\rightarrow0$):}

We shall pass to the limit in \eqref{app1} for $\delta\rightarrow0$
by using Theorem 2.2 from \cite{BBHT}. Note that Fatou's lemma (after
passing on to an a.e. convergent subsequence) and $\Phi_{\lambda}(J_{\delta}\cdot)\le\Phi_{\lambda}(\cdot)$
(which holds e.g. by \cite[Example 7.11]{GessToelle14}) imply that
$\Phi_{\lambda}\circ J_{\delta}\longrightarrow\Phi_{\lambda}$ in
Mosco sense as $\delta\to0$ (for the terminology, see \cite{A84}).
Therefore, we have that 
\[
J_{\delta}\operatorname{div}\Psi^{\lambda}\left(\nabla J_{\delta}\left(\cdot\right)\right)+\varepsilon\Delta\left(\cdot\right)\overset{G}{\longrightarrow}\operatorname{div}\Psi^{\lambda}\left(\nabla\left(\cdot\right)\right)+\varepsilon\Delta\left(\cdot\right),
\]
as $\delta\to0$ and so for the corresponding inverse subdifferential
operators\footnote{Note that, $\Phi^{n}\to\Phi$ in Mosco sense implies that $\partial\Phi^{n}\to\partial\Phi$
in $G$-sense and $(\partial\Phi^{n})^{-1}\to(\partial\Phi)^{-1}$
in $G$-sense, see \cite{nonlin,A84}.}. Also, it is clear that for $v\in S$, 
\[
\nabla(J_{\delta}v)\longrightarrow\nabla v,\quad\text{strongly in }L^{2}(\Ocal;\R^{d}),
\]
as $\delta\to0$, which is sufficient for the strong convergence of
the $C_{0}$-groups of linear operators associated to $B_{i}^{\delta}$
to the $C_{0}$-group associated to $B_{i}$, see e.g. \cite{Brezis:1972wm}.
Therefore, we can apply \cite[Theorems 2.2 and 2.3]{BBHT} (note that
we do not need that the semigroups converge in $C^{1}([0,T];H)$,
as we do not assume any time dependence for our noise coefficients)
and obtain that $\P$-a.s. as $\delta\to0$, 
\begin{eqnarray*}
X^{\varepsilon,\lambda,\delta} & \longrightarrow & X^{\varepsilon,\lambda},\quad\text{weakly in }L^{2}\left([0,T];S\right)\text{ and}\\
 & \phantom{\longrightarrow} & \phantom{X^{\varepsilon,\lambda},}\quad\text{\,\ weakly\ensuremath{^{\ast}}in }L^{\infty}\left([0,T];H\right).
\end{eqnarray*}

Combining with \eqref{H2}, we get by weak lower semicontinuity of
the norm that 
\[
\underset{t\in\left[0,T\right]}{\operatorname{ess\,sup}\displaylimits\,}\mathbb{E}\left\Vert \nabla X_{t}^{\varepsilon,\lambda}\right\Vert _{L^{2}(\Ocal;\R^{d})}^{2}+2\varepsilon\mathbb{E}\int_{0}^{T}\left\Vert \Delta X_{s}^{\varepsilon,\lambda}\right\Vert _{H}^{2}ds\le\E\left\Vert \nabla x_{0}\right\Vert _{L^{2}(\Ocal;\R^{d})}^{2}.
\]

We have proved that there exists a strong solution (in the sense of
\cite{BBHT}) to 
\begin{equation}
\left\{ \begin{array}{ll}
dX_{t}^{\varepsilon,\lambda}=\operatorname{div}\Psi^{\lambda}\left(\nabla X_{t}^{\varepsilon,\lambda}\right)\,dt+\varepsilon\Delta X_{t}^{\varepsilon,\lambda}\,dt & ~\text{in}~\left(0,T\right)\times\mathcal{O},\\
\quad\quad\quad+\dfrac{1}{2}\sum\limits _{i=1}^{N}B_{i}^{2}(X_{t}^{\varepsilon,\lambda})dt+\sum\limits _{i=1}^{N}B_{i}(X_{t}^{\varepsilon,\lambda})\,d\beta_{t}^{i},\\
\medskip X_{0}^{\varepsilon,\lambda}=x_{0}, & ~\text{in}~\mathcal{O},\\
\dfrac{\partial X_{t}^{\varepsilon,\lambda}}{\partial\nu}=0, & ~\text{on}~\left(0,T\right)\times\partial\mathcal{O},
\end{array}\right.\label{app1-1}
\end{equation}
for initial datum $x_{0}\in L^{2}(\Omega;S)$ which is of the particular
form \eqref{eq:test-eq} as claimed in Remark \ref{rem:test-func-rem}.

\emph{Step III ($\lambda\rightarrow0$):}

By applying the Itô formula with $u\mapsto\frac{1}{2}\left\Vert u\right\Vert _{H}^{2}$
and the expectation to the difference 
\begin{eqnarray*}
d\left(X_{t}^{\varepsilon,\lambda}-Z_{t}\right) & = & \left(\operatorname{div}\left(\Psi^{\lambda}\left(\nabla X_{t}^{\varepsilon,\lambda}\right)\right)+\varepsilon\Delta X_{t}^{\varepsilon,\lambda}-G_{t}\right)dt\\
 &  & +\frac{1}{2}\sum\limits _{i=1}^{N}\left(B_{i}^{2}(X_{t}^{\varepsilon,\lambda})-B_{i}^{2}(Z_{t})\right)dt+\sum_{i=1}^{N}B_{i}\left(X_{t}^{\varepsilon,\lambda}-Z_{t}\right)d\beta_{t}^{i},
\end{eqnarray*}
for $Z$ and $G$ considered as in Definition \ref{def:SVI}, we see
that $X^{\varepsilon,\lambda}$ is also a SVI solution to \eqref{app1},
i.e. 
\begin{eqnarray}
 &  & \frac{1}{2}\mathbb{E}\left\Vert X_{t}^{\varepsilon,\lambda}-Z_{t}\right\Vert _{H}^{2}+\mathbb{E}\int_{0}^{t}\int_{\mathcal{O}}\varphi^{\lambda}\left(\nabla X_{s}^{\varepsilon,\lambda}\right)d\xi ds\label{defap2}\\
 &  & +\varepsilon\mathbb{E}\int_{0}^{t}\int_{\mathcal{O}}\left\langle \nabla X_{s}^{\varepsilon,\lambda},\nabla X_{s}^{\varepsilon,\lambda}-\nabla Z_{s}\right\rangle d\xi ds\notag\\
 & \leq & \frac{1}{2}\mathbb{E}\left\Vert x_{0}-Z_{0}\right\Vert _{H}^{2}+\mathbb{E}\int_{0}^{t}\int_{\mathcal{O}}\varphi^{\lambda}\left(\nabla Z_{s}\right)d\xi ds\notag\\
 &  & -\mathbb{E}\int_{0}^{t}\int_{\mathcal{O}}G_{s}\left(X_{s}^{\varepsilon,\lambda}-Z_{s}\right)d\xi ds.\notag
\end{eqnarray}

In order to pass to the limit we shall need the following a-priori
estimates.

First we apply the Itô formula for the functional $u\mapsto\frac{1}{2}\left\Vert u\right\Vert _{H}^{2}$
to the equation 
\begin{equation}
\left\{ \begin{array}{ll}
dX_{t}^{\varepsilon,\lambda}=\operatorname{div}\Psi^{\lambda}\left(\nabla X_{t}^{\varepsilon,\lambda}\right)dt+\varepsilon\Delta X_{t}^{\varepsilon,\lambda}dt & ~\text{in}~\left(0,T\right)\times\mathcal{O}\\
\quad\quad\quad\quad\quad+\dfrac{1}{2}\sum\limits _{i=1}^{N}B_{i}^{2}(X_{t}^{\varepsilon,\lambda})dt+\sum\limits _{i=1}^{N}B_{i}\left(X_{t}^{\varepsilon,\lambda}\right)d\beta_{t}^{i},\\
\medskip X_{0}^{\varepsilon,\lambda}=x_{0}, & ~\text{in}~\mathcal{O}\\
\dfrac{\partial X_{t}^{\varepsilon,\lambda}}{\partial\nu}=0, & ~\text{on}~\left(0,T\right)\times\partial\mathcal{O}
\end{array}\right.\label{app2}
\end{equation}
in order to get that 
\begin{eqnarray}
 &  & \frac{1}{2}\mathbb{E}\left\Vert X_{t}^{\varepsilon,\lambda}\right\Vert _{H}^{2}+\mathbb{E}\int_{0}^{t}\int_{\mathcal{O}}\varphi^{\lambda}\left(\nabla X_{s}^{\varepsilon,\lambda}\right)d\xi ds+\varepsilon\mathbb{E}\int_{0}^{t}\int_{\mathcal{O}}\left\vert \nabla X_{s}^{\varepsilon,\lambda}\right\vert ^{2}d\xi ds\label{eq:lambdaepsilonito}\\
 & \leq & \frac{1}{2}\E\left\Vert x_{0}\right\Vert _{H}^{2},\quad\forall\lambda,~t\in\left[0,T\right].\nonumber 
\end{eqnarray}

Moreover, in order to verify \eqref{eq:reg}, we see that by the Mosco
convergence (see e.g. \cite[Proposition 6.2]{GessToelle15}) 
\[
\Phi_{\lambda}=\int_{\Ocal}\varphi^{\lambda}(\nabla\cdot)\,d\xi\longrightarrow\Phi\quad\text{in Mosco sense as \ensuremath{\lambda\to0}},
\]
and Fatou's lemma (after passing to an a.e. convergent subsequence
--- strong $L^{2}$-convergence is justified below), we get that 
\[
\mathbb{E}\int_{0}^{t}\Phi\left(X_{s}^{\varepsilon}\right)ds\leq\liminf_{\lambda\to0}\mathbb{E}\int_{0}^{t}\int_{\mathcal{O}}\varphi^{\lambda}\left(\nabla X_{s}^{\varepsilon,\lambda}\right)d\xi ds<\infty.
\]
On the other hand, also by the Itô formula and Lemma \ref{lem:B-properties}
(i), we get that $\P$-a.s., $t\in[0,T]$, 
\begin{eqnarray*}
 &  & \frac{1}{2}\left\Vert X_{t}^{\varepsilon,\lambda_{1}}-X_{t}^{\varepsilon,\lambda_{2}}\right\Vert _{H}^{2}\\
 &  & +\int_{0}^{t}\left(\Psi^{\lambda_{1}}\left(\nabla X_{s}^{\varepsilon,\lambda_{1}}\right)-\Psi^{\lambda_{2}}\left(\nabla X_{s}^{\varepsilon,\lambda_{2}}\right),\nabla X_{s}^{\varepsilon,\lambda_{1}}-\nabla X_{s}^{\varepsilon,\lambda_{2}}\right)_{L^{2}(\Ocal;\R^{d})}ds\\
 &  & +\varepsilon\int_{0}^{t}\left\Vert \nabla X_{s}^{\varepsilon,\lambda_{1}}-\nabla X_{s}^{\varepsilon,\lambda_{2}}\right\Vert _{L^{2}(\Ocal;\R^{d})}^{2}ds\\
 & = & \sum\limits _{i=1}^{N}\int_{0}^{t}\left(B_{i}\left(X_{s}^{\varepsilon,\lambda_{1}}-X_{s}^{\varepsilon,\lambda_{2}}\right)d\beta_{s}^{i}~,X_{s}^{\varepsilon,\lambda_{1}}-X_{s}^{\varepsilon,\lambda_{2}}\right)_{H}.
\end{eqnarray*}

By using \cite[eq. (A.6) in Appendix A]{GessToelle15}, we have for
all $\xi,\zeta\in\R^{d}$ and some positive constant $C>0$ that

\[
\langle\Psi^{\lambda_{1}}(\xi)-\Psi^{\lambda_{2}}(\zeta),\xi-\zeta\rangle\geq-C\left(\lambda_{1}+\lambda_{2}\right)\left(1+\left\vert \xi\right\vert ^{2}+\left\vert \zeta\right\vert ^{2}\right).
\]
We obtain $\P\otimes ds$-a.s. that 
\begin{eqnarray*}
 &  & \left(\Psi^{\lambda_{1}}\left(\nabla X^{\varepsilon,\lambda_{1}}\right)-\Psi^{\lambda_{2}}\left(\nabla X^{\varepsilon,\lambda_{2}}\right),\nabla X^{\varepsilon,\lambda_{1}}-\nabla X^{\varepsilon,\lambda_{2}}\right)_{L^{2}\left(\mathcal{O};\mathbb{R}^{d}\right)}\\
 & \geq & -C\left(\lambda_{1}+\lambda_{2}\right)\int\limits _{\mathcal{O}}\left(1+\left\vert \nabla X^{\varepsilon,\lambda_{1}}\right\vert ^{2}+\left\vert \nabla X^{\varepsilon,\lambda_{2}}\right\vert ^{2}\right)d\xi\\
 & \geq & -C\left(\lambda_{1}+\lambda_{2}\right)\left(1+\left\vert X^{\varepsilon,\lambda_{1}}\right\vert _{S}^{2}+\left\vert X^{\varepsilon,\lambda_{2}}\right\vert _{S}^{2}\right),
\end{eqnarray*}
and then, by \eqref{H2} and \eqref{eq:lambdaepsilonito}, we get
for the expectation, that 
\[
\mathbb{E}\int\limits _{0}^{t}\left(\Psi^{\lambda_{1}}\left(\nabla X_{s}^{\varepsilon,\lambda_{1}}\right)-\Psi^{\lambda_{2}}\left(\nabla X_{s}^{\varepsilon,\lambda_{2}}\right),\nabla X_{s}^{\varepsilon,\lambda_{1}}-\nabla X_{s}^{\varepsilon,\lambda_{2}}\right)_{L^{2}\left(\mathcal{O};\mathbb{R}^{d}\right)}ds\geq-C\left(\lambda_{1}+\lambda_{2}\right).
\]

Now, by the Burkholder-Davis-Gundy inequality, taking \eqref{eq:uBu}
into account, and by the above computation concerning $\Psi^{\lambda}$,
we obtain that for all $t\in[0,T]$, 
\[
\frac{{1}}{2}\mathbb{E}\underset{0\leq s\leq t}{\sup}\left\Vert X_{s}^{\varepsilon,\lambda_{1}}-X_{s}^{\varepsilon,\lambda_{2}}\right\Vert _{H}^{2}+\varepsilon\E\int_{0}^{t}\left\Vert \nabla X_{s}^{\varepsilon,\lambda_{1}}-\nabla X_{s}^{\varepsilon,\lambda_{2}}\right\Vert _{H}^{2}ds\leq C(\lambda_{1}+\lambda_{2}).
\]

Consequently, we have that 
\begin{equation}
\underset{\lambda\rightarrow0}{\lim}\mathbb{E}\left[\underset{t\in\left[0,T\right]}{\sup}\left\Vert X_{t}^{\varepsilon,\lambda}-X_{t}^{\varepsilon}\right\Vert _{H}^{2}\right]=0.\label{eq:lambda-eps}
\end{equation}

We can now pass to the limit for $\lambda\rightarrow0$ in \eqref{defap2}
in order to obtain (recall that $\int_{\Ocal}\varphi^{\lambda}(\cdot)\,d\xi\le\Phi$)
\begin{eqnarray}
 &  & \frac{1}{2}\mathbb{E}\left\Vert X_{t}^{\varepsilon}-Z_{t}\right\Vert _{H}^{2}+\mathbb{E}\int_{0}^{t}\int_{\mathcal{O}}\varphi\left(\nabla X_{s}^{\varepsilon}\right)d\xi ds\label{defap1}\\
 &  & +\varepsilon\mathbb{E}\int_{0}^{t}\int_{\mathcal{O}}\left\langle \nabla X_{s}^{\varepsilon},\nabla X_{s}^{\varepsilon}-\nabla Z_{s}\right\rangle d\xi ds\notag\\
 & \leq & \frac{1}{2}\mathbb{E}\left\Vert x_{0}-Z_{0}\right\Vert _{H}^{2}+\mathbb{E}\int_{0}^{t}\int_{\mathcal{O}}\varphi\left(\nabla Z_{s}\right)d\xi ds\notag\\
 &  & -\mathbb{E}\int_{0}^{t}\int_{\mathcal{O}}G_{s}\left(X_{s}^{\varepsilon}-Z_{s}\right)d\xi ds.\notag
\end{eqnarray}

\emph{Step IV ($\varepsilon\rightarrow0$):}

Arguing as in the previous step, we get that 
\begin{eqnarray}
 &  & \frac{1}{2}\mathbb{E}\left\Vert X_{t}^{\varepsilon}\right\Vert _{H}^{2}+\mathbb{E}\int_{0}^{t}\int_{\mathcal{O}}\varphi\left(\nabla X_{s}^{\varepsilon}\right)d\xi ds+\varepsilon\mathbb{E}\int_{0}^{t}\int_{\mathcal{O}}\left\vert \nabla X_{s}^{\varepsilon}\right\vert ^{2}d\xi ds\medskip\label{eq:reg_L2}\\
 & \leq & \frac{1}{2}\E\left\Vert x_{0}\right\Vert _{H}^{2},\quad t\in\left[0,T\right].\nonumber 
\end{eqnarray}

By Itô's formula, this time, considering the process for fixed $\lambda>0$
and not $\varepsilon>0$ fixed, as previously, and by monotonicity,
we get that 
\begin{eqnarray*}
 &  & \frac{1}{2}\left\Vert X_{t}^{\varepsilon_{1},\lambda}-X_{t}^{\varepsilon_{2},\lambda}\right\Vert _{H}^{2}e^{-t}\\
 &  & +\int_{0}^{t}e^{-s}\left(\varepsilon_{1}\nabla X_{s}^{\varepsilon_{1},\lambda}-\varepsilon_{2}\nabla X_{s}^{\varepsilon_{2},\lambda},\nabla X_{s}^{\varepsilon_{1},\lambda}-\nabla X_{s}^{\varepsilon_{2},\lambda}\right)_{L^{2}(\Ocal;\R^{d})}ds\\
 &  & +\frac{1}{2}\int_{0}^{t}e^{-s}\|X_{s}^{\varepsilon_{1},\lambda}-X_{s}^{\varepsilon_{2},\lambda}\|_{H}^{2}\,ds\\
 & \leq & \sum\limits _{i=1}^{N}\int_{0}^{t}e^{-s}\left(B_{i}\left(X_{s}^{\varepsilon_{1},\lambda}-X_{s}^{\varepsilon_{2},\lambda}\right)d\beta_{s}^{i}~,X_{s}^{\varepsilon_{1},\lambda}-X_{s}^{\varepsilon_{2},\lambda}\right)_{H}
\end{eqnarray*}
Since $\P\otimes ds$-a.s., 
\begin{eqnarray*}
 &  & \left(\varepsilon_{1}\nabla X^{\varepsilon_{1},\lambda}-\varepsilon_{2}\nabla X^{\varepsilon_{2},\lambda},\nabla X^{\varepsilon_{1},\lambda}-\nabla X^{\varepsilon_{2},\lambda}\right)_{L^{2}(\Ocal;\R^{d})}\medskip\\
 & = & -\left(\varepsilon_{1}\Delta X^{\varepsilon_{1},\lambda}-\varepsilon_{2}\Delta X^{\varepsilon_{2},\lambda},X^{\varepsilon_{1},\lambda}-X^{\varepsilon_{2},\lambda}\right)_{H}\medskip\\
 & \geq & -\frac{1}{2}\left(\varepsilon_{1}^{2}\left\Vert \Delta X^{\varepsilon_{1},\lambda}\right\Vert _{H}^{2}+\varepsilon_{2}^{2}\left\Vert \Delta X^{\varepsilon_{2},\lambda}\right\Vert _{H}^{2}\right)-\frac{1}{2}\left\Vert X^{\varepsilon_{1},\lambda}-X^{\varepsilon_{2},\lambda}\right\Vert _{H}^{2}
\end{eqnarray*}
and by using again the Burkholder-Davis-Gundy inequality and \eqref{eq:uBu},
we get that for all $t\in[0,T]$, 
\begin{eqnarray*}
 &  & \mathbb{E}\underset{0\leq s\leq t}{\sup}e^{-s}\left\Vert X_{s}^{\varepsilon_{1},\lambda}-X_{s}^{\varepsilon_{2},\lambda}\right\Vert _{H}^{2}\\
 & \leq & \varepsilon_{1}^{2}\mathbb{E}\int_{0}^{t}e^{-s}\left\Vert \Delta X_{s}^{\varepsilon_{1},\lambda}\right\Vert _{H}^{2}ds+\varepsilon_{2}^{2}\mathbb{E}\int_{0}^{t}e^{-s}\left\Vert \Delta X_{s}^{\varepsilon_{2},\lambda}\right\Vert _{H}^{2}ds.
\end{eqnarray*}

Keeping in mind that for initial data in $L^{2}(\Omega;S)$, by \emph{Step
II} above, in particular, by \eqref{H2}, 
\begin{equation}
\varepsilon\mathbb{E}\int_{0}^{T}\left\Vert \Delta X_{s}^{\varepsilon,\lambda}\right\Vert _{H}^{2}ds\le C,\label{H22}
\end{equation}
uniformly in $\lambda>0$, we obtain by \eqref{eq:lambda-eps}, for
letting first $\lambda\to0$, 
\[
\underset{\varepsilon\rightarrow0}{\lim}\mathbb{E}\left[\underset{t\in\left[0,T\right]}{\sup}\left\Vert X_{t}^{\varepsilon}-X_{t}\right\Vert _{H}^{2}\right]=0
\]
for some limiting process $X\in C([0,T];L^{2}(\Omega;H))$. Note that
by weak convergence in $L^{2}([0,T];S)$, we get that $\P\otimes dt$-a.s.
$X\in H^{1}(\Ocal)$.

Finally, by computing $\P\otimes ds$-a.e. 
\begin{eqnarray*}
\varepsilon\int_{\mathcal{O}}\left\langle \nabla X^{\varepsilon},\nabla X^{\varepsilon}-\nabla Z\right\rangle d\xi & = & -\varepsilon\int_{\mathcal{O}}\Delta X^{\varepsilon}\left(X^{\varepsilon}-Z\right)d\xi\\
 & \geq & -\frac{1}{2}\varepsilon^{\frac{{4}}{3}}\left\Vert \Delta X^{\varepsilon}\right\Vert _{H}^{2}-\frac{1}{2}\varepsilon^{\frac{{2}}{3}}\left\Vert X^{\varepsilon}-Z\right\Vert _{H}^{2}\medskip
\end{eqnarray*}
and using again \eqref{H22} we can pass to the limit in \eqref{defap1}
and get that $X$ is a continuous SVI solution (which satisfies \eqref{eq:reg}
by passing to the limit in \eqref{eq:reg_L2}) 
\begin{eqnarray*}
 &  & \frac{1}{2}\mathbb{E}\left\Vert X_{t}-Z_{t}\right\Vert _{H}^{2}+\mathbb{E}\int_{0}^{t}\int_{\mathcal{O}}\varphi\left(\nabla X_{s}\right)d\xi ds\\
 & \leq & \frac{1}{2}\mathbb{E}\left\Vert x_{0}-Z_{0}\right\Vert _{H}^{2}+\mathbb{E}\int_{0}^{t}\int_{\mathcal{O}}\varphi\left(\nabla Z_{s}\right)d\xi ds\\
 &  & -\mathbb{E}\int_{0}^{t}\int_{\mathcal{O}}G_{s}\left(X_{s}-Z_{s}\right)d\xi ds.
\end{eqnarray*}

\emph{Step V (general initial conditions):}

In order to conclude the proof of existence we only need to extend
the solution for arbitrary $x\in L^{2}(\Omega,\Fcal_{0},\P;H)$. Let
$X,X^{\ast}$ be continuous SVI solutions starting in $x,x^{\ast}$,
resp. Note that $S$ is dense in $H$, so this follows directly from
\[
\mathbb{E}\left[\underset{t\in\left[0,T\right]}{\sup}\left\Vert X_{t}-X_{t}^{\ast}\right\Vert _{H}^{2}\right]\leq\E\left\Vert x_{0}-x_{0}^{\ast}\right\Vert _{H}^{2},
\]
which can easily be obtained by arguments similar to those from the
previous steps (using monotonicity, Burkholder-Davis-Gundy and \eqref{eq:uBu}).

\end{proof}

\subsection{\label{sub:Uniqueness}Uniqueness}
\begin{proof}[Proof of Theorem \ref{thm:main_thm} (continued).]
 The existence of a continuous SVI solution is proved in the section
above. We follow an argument from \cite{GessToelle15}. Let $X$ be
any continuous SVI solution to \eqref{eq:main_equation} with initial
condition $x\in L^{2}(\Omega,\Fcal_{0},\P;H)$ and let $Y^{\varepsilon,\lambda,n}$
be the strong approximating solution to \eqref{app1-1} with initial
condition $y^{n}\in L^{2}(\Omega,\Fcal_{0},\P;S)$ such that $y^{n}\to y$
in $L^{2}(\Omega,\Fcal_{0},\P;H)$, where $y\in L^{2}(\Omega,\Fcal_{0},\P;H)$.
Then the following variational inequality holds (with $Z_{0}=y^{n}$,
$Z=Y^{\varepsilon,\lambda,n}$, $G=\operatorname{div}\Psi^{\lambda}\left(\nabla Y^{\varepsilon,\lambda,n}\right)+\varepsilon\Delta Y^{\varepsilon,\lambda,n}$),
\[
\begin{split}\frac{{1}}{2}\E\|X_{t}-Y_{t}^{\varepsilon,\lambda,n}\|_{H}^{2} & +\E\int_{0}^{t}\Phi(X_{s})\,ds\\
 & \le\frac{{1}}{2}\E\|x-y^{n}\|_{H}^{2}+\E\int_{0}^{t}\Phi(Y_{s}^{\varepsilon,\lambda,n})\,ds\\
 & -\E\int_{0}^{t}(\varepsilon\Delta Y_{s}^{\varepsilon,\lambda,n}+\div\Psi^{\lambda}(\nabla Y_{s}^{\varepsilon,\lambda,n}),X_{s}-Y_{s}^{\varepsilon,\lambda,n})_{H}\,ds,
\end{split}
\]
for a.e. $t\in[0,T]$.

By \cite[Appendix A]{GessToelle15} for all $z\in H^{1}$ we have
\[
-(\div\Psi^{\lambda}(\nabla Y^{\varepsilon,\lambda,n}),z-Y^{\varepsilon,\lambda,n})_{H}+\Phi(Y^{\varepsilon,\lambda,n})\le\Phi(z)+C\lambda(1+\Phi(Y^{\varepsilon,\lambda,n}))\quad ds\otimes\P-\text{a.e}.
\]
Since $\Phi$ is the lower-semicontinuous envelope of $\tilde{{\Phi}}=\Phi\vert_{H^{1}}$
(i.e., $\Phi$ restricted to $H^{1}$), for $ds\otimes\P$-a.e. $(s,\omega)\in[0,T]\times\Omega$,
we can choose a sequence $z^{m}\in H^{1}$ such that $z^{m}\to X_{s}(\omega)$
in $H$ and $\Phi(z^{m})\to\Phi(X_{s}(\omega))$.

Hence, 
\[
-(\div\Psi^{\lambda}(\nabla Y^{\varepsilon,\lambda,n}),X-Y^{\varepsilon,\lambda,n})_{H}+\Phi(Y^{\varepsilon,\lambda,n})\le\Phi(X)+C\lambda(1+\Phi(Y^{\varepsilon,\lambda,n}))\quad ds\otimes\P-\text{a.e}.
\]
Thus, 
\[
\begin{split}\frac{{1}}{2}\E\|X_{t}-Y_{t}^{\varepsilon,\lambda,n}\|_{H}^{2} & \le\frac{{1}}{2}\E\|x-y^{n}\|_{H}^{2}+C\lambda\E\int_{0}^{t}\left(1+\Phi(Y_{s}^{\varepsilon,\lambda,n})\right)\,ds\\
 & +\frac{1}{2}\E\int_{0}^{t}\left(\varepsilon^{\frac{{4}}{3}}\|\Delta Y_{s}^{\varepsilon,\lambda,n}\|_{H}^{2}+\varepsilon^{\frac{{2}}{3}}\|X_{s}-Y_{s}^{\varepsilon,\lambda,n}\|_{H}^{2}\right)\,ds.
\end{split}
\]
Taking first $\lambda\to0$ then $\varepsilon\to0$ (using the $H^{2}$
bound \eqref{H22}, which is uniform in $\lambda$) and then $n\to\infty$
yields 
\begin{align*}
\E\|X_{t}-Y_{t}\|_{H}^{2}\le & \E\|x-y\|_{H}^{2},
\end{align*}
for a.e. $t\in[0,T]$.

\end{proof}
\appendix

\section{\label{sec:Vector-fields}Vector fields of Killing}

In this section, suppose that $\Ocal\subset\R^{d}$, open, bounded,
with smooth boundary $\partial\Ocal$.
\begin{defn}
A $C^{1}$-vector field $b:\mathcal{\overline{O}}\to\mathbb{R}^{d}$,
is called \emph{Killing vector field}, if the following condition
is satisfied on $\overline{\mathcal{O}}$
\begin{equation}
\partial_{j}b^{i}+\partial_{i}b^{j}=0\quad\forall1\le i,j\le d.\label{eq:kill_basic}
\end{equation}

\end{defn}
The following lemma is based on ideas from \cite{Yano:1965ke}, see
also \cite{Yano:1959bm}.
\begin{lem}
\label{lem:killing_lemma}A (sufficiently smooth) vector field $b:\overline{\mathcal{O}}\to\mathbb{R}^{d}$
is a Killing vector field if and only if
\begin{equation}
\langle\Delta b,b\rangle=0,\quad\operatorname{div}b=0\quad\text{on\;\ensuremath{\overline{\ensuremath{\mathcal{O}}}}.}\label{eq:killing1}
\end{equation}

\begin{equation}
\sum_{1\le i,j\le d}(\partial_{j}b^{i}+\partial_{i}b^{j})\nu_{j}b^{i}=0\quad\text{on\;\ensuremath{\partial\mathcal{O}}},\label{eq:killing2}
\end{equation}
where the Laplace acts componentwise and where $\nu:\partial\mathcal{O}\to\mathbb{R}^{d}$
denotes the outer unit normal of $\mathcal{\overline{O}}$.\end{lem}
\begin{proof}
One easily sees that,
\[
\begin{aligned} & \partial_{j}[(\partial_{j}b^{i}+\partial_{i}b^{j})b^{i}-b^{j}(\partial_{i}b^{i})]\\
= & (\partial_{j}\partial_{j}b^{i})b^{i}+(\partial_{j}b^{i})(\partial_{j}b^{i})\\
 & +(\partial_{j}\partial_{i}b^{j})b^{i}+(\partial_{i}b^{j})(\partial_{j}b^{i})\\
 & -(\partial_{j}b^{j})(\partial_{i}b^{i})-b^{j}(\partial_{j}\partial_{i}b^{i}).
\end{aligned}
\]

Assuming the above conditions, interchanging the order of differentiation,
and summing over $1\le i,j\le d$, we obtain that
\[
\begin{aligned} & \sum_{1\le i,j\le d}\partial_{j}[(\partial_{j}b^{i}+\partial_{i}b^{j})b^{i}-b^{j}(\partial_{i}b^{i})]\\
= & \langle\Delta b,b\rangle+\sum_{1\le i,j\le d}[(\partial_{j}b^{i})^{2}+(\partial_{i}b^{j})(\partial_{j}b^{i})]-(\operatorname{div}b)^{2}\\
 & +\sum_{1\le i,j\le d}(\partial_{j}\partial_{i}b^{j})b^{i}-(\partial_{i}\partial_{j}b^{i})b^{j}\\
= & \frac{{1}}{2}\sum_{1\le i,j\le d}(\partial_{j}b^{i})^{2}+2(\partial_{i}b^{j})(\partial_{j}b^{i})+(\partial_{i}b^{j})^{2}.\\
= & \frac{{1}}{2}\sum_{1\le i,j\le d}(\partial_{j}b^{i}+\partial_{i}b^{j})^{2}.
\end{aligned}
\]
Denote by $S^{d-1}$ the surface element on $\partial\Ocal$. By Gauss's
divergence theorem, we get that, 
\[
\begin{aligned} & \frac{{1}}{2}\sum_{1\le i,j\le d}\int_{\mathcal{\overline{O}}}(\partial_{j}b^{i}+\partial_{i}b^{j})^{2}\,d\xi\\
= & \sum_{1\le i,j\le d}\int_{\partial\mathcal{O}}[(\partial_{j}b^{i}+\partial_{i}b^{j})b^{i}-b^{j}(\partial_{i}b^{i})]\nu_{j}\,dS^{d-1}\\
= & \sum_{1\le i,j\le d}\int_{\partial\mathcal{O}}[(\partial_{j}b^{i}+\partial_{i}b^{j})b^{i}]\nu_{j}\,dS^{d-1}\\
= & 0,
\end{aligned}
\]
and hence $\partial_{j}b^{i}+\partial_{i}b^{j}=0$ on $\mathcal{\overline{O}}$.

Suppose conversely, that $b$ is a Killing vector field. Then \eqref{eq:killing2}
is automatically satisfied. Also, $\operatorname{div}b=0$ by choosing
$i=j$. Clearly, also
\[
0=b^{i}\partial_{j}\partial_{j}b^{i}+b^{i}\partial_{j}\partial_{i}b^{j}=b^{i}\partial_{j}\partial_{j}b^{i}+b^{i}\partial_{i}\partial_{j}b^{j},
\]
and summing over $1\le i,j\le d$ yields,
\[
\langle\Delta b,b\rangle=0,
\]
and hence \eqref{eq:killing1} is satisfied, too.\end{proof}
\begin{thm}
\label{thm:commu} Let $b:\overline{\Ocal}\to\R^{d}$ be a $C^{1}$-vector
field. In order that the first order differential operator $u\mapsto\langle b,\nabla u\rangle$
commutes with the Laplace operator $u\mapsto-\Delta u$ on the space
of smooth functions on $\overline{\Ocal}$ it is necessary and sufficient
that $b$ is a Killing vector field.\end{thm}
\begin{proof}
See \cite[Theorem 2.1]{Sumitomo:1972if}.
\end{proof}

\specialsection*{Acknowledgements}

\emph{The first author would like to thank Viorel Barbu and Witold
Respondek for helpful comments. The second author would like to thank
the colleagues at the Laboratoire de Mathématiques de l'INSA de Rouen
for a pleasant stay at their department where part of this work was
done. He also would like to express his gratitude to Benjamin Gess,
who made some important remarks on a preliminary version of this work.
Both authors are grateful for some clarifying remarks by the referee.}

\def\cprime{$'$}

\end{document}